\documentclass[12pt]{amsart}
\usepackage{enumerate}
\usepackage[hyperindex,breaklinks,colorlinks,citecolor=blue,pagebackref]{hyperref}
\usepackage[latin1]{inputenc}
\usepackage{amsmath, amsthm, amsfonts, amssymb}
\usepackage{mathrsfs}
\usepackage{graphicx}
\usepackage{latexsym}
\usepackage{fullpage}
\usepackage[all]{xy}
\usepackage{color}
\usepackage{stmaryrd}

\newtheorem{theorem}{Theorem}[section]
\newtheorem{lemma}[theorem]{Lemma}
\newtheorem{proposition}[theorem]{Proposition}
\newtheorem{corollary}[theorem]{Corollary}
\newtheorem{question}[theorem]{Question}
\newtheorem{definition}[theorem]{Definition}

\theoremstyle{remark}
\newtheorem{remark}{Remark}[section]
\newtheorem{rmks}{Remarks}[section]
\newtheorem*{Examples}{Examples}
\newtheorem{example}{Example}
\newtheorem*{claim}{Claim}

\newcommand{\B}{\mathcal{B}}

\newcommand{\D}{\mathscr{D}}

\newcommand{\M}{\mathscr{M}}

\newcommand{\R}{\mathsf{R}}
\renewcommand{\S}{\mathcal{S}}

\newcommand{\U}{\mathcal{U}}

\newcommand{\X}{\mathcal{X}}

\newcommand{\MLR}{\mathsf{MLR}}
\newcommand{\SR}{\mathsf{SR}}

\newcommand{\WtwoR}{\mathsf{W2R}}
\newcommand{\twoMLR}{\mathsf{2MLR}}

\newcommand{\dom}{\text{dom}}

\newcommand{\llb}{\llbracket}
\newcommand{\rrb}{\rrbracket}
\newcommand{\LRr}{\leq_{LR}}
\newcommand{\LRe}{\equiv_{LR}}
\newcommand{\LRD}{\D_{LR}}
\newcommand{\LRrmu}{\leq_{LR(\mu)}}
\newcommand{\LRDmu}{\D_{LR(\mu)}}
\newcommand{\LRDnu}{\D_{LR(\nu)}}
\newcommand{\aand}{\;\&\;}
\newcommand{\RS}{\mathsf{RScope}}
\newcommand{\Deg}{\mathit{deg}}

\definecolor{purple}{rgb}{.9,0.2,.9}

\newcommand{\cs}{2^\omega}
\newcommand{\uh}{{\upharpoonright}}
\renewcommand{\phi}{\varphi}
\newcommand{\str}{2^{<\omega}}
\newcommand{\binrat}{\mathbb{Q}_2}
\newcommand{\appmu}{\widehat{\mu}}

\newcommand{\halts}{{\downarrow}}
\newcommand{\diverge}{{\uparrow}}

\newcommand{\atom}{\mathsf{Atoms}}

\newcommand{\twopartdefow}[4]
{
	\left\{
		\begin{array}{ll}
			#1 & \mbox{if } #2 \\
			#3 & \mbox{otherwise} #4
		\end{array}
	\right.
}

\title{Trivial measures are not so trivial}
\author{Christopher P. Porter}
\date{} 

\begin{document}

\maketitle

\begin{abstract}
Although algorithmic randomness with respect to various non-uniform computable measures is well-studied, little attention has been paid to algorithmic randomness with respect to computable \emph{trivial} measures, where a measure $\mu$ on $\cs$ is trivial if the support of $\mu$ consists of a countable collection of sequences.  In this article, it is shown that there is much more structure to trivial computable measures than has been previously suspected.  
\end{abstract}

\section{Introduction}\label{intro}

Algorithmic randomness with respect to non-uniform computable measures is a well-studied subject.  Although one can find definitions of biased random sequences is the work of von Mises, Ville, and Church, the first generally accepted definition of biased algorithmic randomness can be found in the final section of Martin-L\"of's groundbreaking 1966 paper, ``The Definition of Random Sequences" \cite{Mar66}, in which Martin-L\"of considers definitions of algorithmic randomness for various Bernoulli measures.  This study of definitions of biased randomness was continued in the 1970s by Levin and Zvonkin \cite{ZvoLev70} and Schnorr \cite{Sch71},\cite{Sch77}.  In particular, Levin, Zvonkin, and Schnorr studied measures that are \emph{atomic}, i.e., measures $\mu$ for which there is some $A\in\cs$ such that $\mu(\{A\})>0$, where $A$ is called an \emph{atom} of $\mu$, denoted $A\in\atom_\mu$.  \\

Schnorr further studied what he called \emph{discrete} measures, where $\mu$ is discrete if~$\mu(\atom_\mu)=1$, or equivalently, if the support of $\mu$ is a countable collection of sequences.  In \cite{Sch77}, Schnorr claimed the following:
\begin{claim}
$\MLR_\mu=\SR_\mu$ if and only if $\mu$ is discrete.
\end{claim}
\noindent
Here $\MLR_\mu$ denotes the collection of $\mu$-Martin-L\"of random sequences and $\SR_\mu$ denotes the collection of $\mu$-Schnorr random sequences.

Discrete measures were later referred to in Kautz's dissertation \cite{Kau91} as \emph{trivial}, the implication being that such measures have no interesting structural properties; once we've assigned all measure to some countable collection of points, there appears to be nothing left to say about the resulting measure.

The main goal of this paper is to show that this is \emph{not} the case; there is much more structure to trivial measures than has been previously suspected.  Not only do we show that the above claim of Schnorr's is false, but we also construct a number of trivial measures that allow us to separate various notions of randomness.  That is, we prove several theorems of the following form: given two randomness notions $\R^1$ and $\R^2$ such that $\R^1_\mu\subseteq\R^2_\mu$ for every computable measure $\mu$, there is a trivial computable measure $\mu$ such that (i) $\R^1_\mu=\atom_\mu$ and (ii) $\R^2_\mu\setminus\R^1_\mu\neq\emptyset$.
In particular, we construct a trivial measure $\mu$ such that $\MLR_\mu=\atom_\mu$ and $\SR_\mu\neq\atom_\mu$.  As further evidence of the non-triviality of trivial measures, we study a degree structure associated with a given trivial $\mu$ called the \emph{$LR(\mu)$-degrees} and show that for each finite distributive lattice $(L,\leq)$, there is a computable trivial measure $\mu$ so that the collection of $LR(\mu)$-degrees is isomorphic to $(L,\leq)$. 

The outline of this paper is as follows.  In Section \ref{sec:background}, we provide the relevant technical background for the rest of the paper.  Next, in Section \ref{sec:tally}, we discuss the main technique for constructing trivial measures that we employ throughout this paper, defining these measures in terms of what we call \emph{tally functionals}.  In Section \ref{sec:separating}, we separate various notions of randomness via trivial measures, including the separation of Martin-L\"of randomness and Schnorr randomness that provides a counterexample to Schnorr's claim.  Lastly, in Section \ref{sec:fdl}, we discuss how to produce a trivial measure $\mu$ with an associated $LR(\mu)$-degree structure that is isomorphic to a given finite distributive lattice.

We assume that the reader is familiar with the basics of computability theory:  computable functions, partial computable functions, computably enumerable sets, Turing functionals, Turing degrees, the Turing jump, and so on (see, for instance, \cite{Soa87}), as well as the basics of effective randomness (otherwise, we refer the reader to~\cite{DowHir10} or ~\cite{Nie09}).  $\cs$ is the set of infinite binary sequences, also known as Cantor space.  $\str$ is the set of finite binary strings. $\binrat$ is the set of dyadic rationals, i.e., multiples of a negative power of~$2$.  Given $X\in \cs$ and an integer~$n$, $X \uh n$ is the string that consists of the first~$n$ bits of~$X$, and $X(n)$ is the $(n+1)$st of $X$ (so that $X(0)$ is the first bit of $X$).  If $\sigma,\tau\in\str$, then $\sigma\preceq\tau$ means that $\sigma$ is an initial segment of $\tau$; moreover, given $X\in\cs$, $\sigma\prec X$ means that $\sigma$ is an initial segment of~$X$. Given a string~$\sigma$, the basic open set determined by $\sigma$ is defined to be $\llb\sigma\rrb=\{X:\sigma\prec X\}.$ For $X,Y\in\cs$, we define $X\oplus Y=\{2n:n\in X\}\cup\{2n+1:n\in Y\}$.  Given a collection $\{B_i\}_{i\in\omega}\subseteq\cs$, we define $\bigoplus_{i\in\omega}B_i=\{\langle n,i\rangle:n\in B_i\}$, where $\langle\cdot,\cdot\rangle$ is some computable pairing function.  For $A\in\cs$, $A^{[i]}=\{n:\langle n,i\rangle\in A\}$, so that $A=\bigoplus_{i\in\omega}A^{[i]}$.
 
\section{Computable Measures and Randomness}\label{sec:background}

In this section, we review the relevant material on computability probability measures on $\cs$ and the various notions of algorithmic randomness given in terms of these measures.

\subsection{Computable measures on $\cs$}

A probability measure on $\cs$ assigns to each Borel subset of $\cs$ a real in $[0,1]$. It suffices to consider the restriction of probability measures to basic open subsets of $\cs$, for Caratheodory's theorem from classical measure theory ensures that a function~$\mu$ defined on basic open sets that satisfies $\mu(\llb\sigma\rrb)=\mu(\llb\sigma0\rrb)+\mu(\llb\sigma1\rrb)$ for all $\sigma\in\str$ can be uniquely extended to a probability measure on $\cs$. We can therefore represent measures as functions from strings to reals, where for all~$\sigma\in\str$, $\mu(\sigma)$ will denote the $\mu$-measure of ~$\llb\sigma\rrb$. This concise representation also allows us to talk about \emph{computable} probability measures. 

\begin{definition}
A probability measure $\mu$ on $\cs$ is \emph{computable} if $\sigma \mapsto \mu(\sigma)$ is computable as a real-valued function, i.e., there is a computable function $\appmu:\str\times\omega\rightarrow\binrat$ such that
\[|\mu(\sigma)-\appmu(\sigma,i)|\leq 2^{-i}\]
for every $\sigma\in\str$ and $i\in\omega$.
\end{definition}

In what follows~$\lambda$ will refer exclusively to the Lebesgue measure on~$\cs$, i.e., $\lambda(\sigma)=2^{-|\sigma|}$ for each $\sigma\in\str$.  Moreover, $\M_c$ denotes the collection of computable measures on $\cs$.

An important technique for defining a computable measure on $\cs$ is to induce the measure by means of some Turing functional $\Phi$.  To do so, it must be the case that $\Phi$ is \emph{almost total}, i.e., $\lambda(\dom(\Phi))=1$.  For our purposes, however, it suffices to restrict to truth-table functionals.

\begin{definition}
A Turing functional $\Phi: \subseteq \cs\rightarrow\cs$ is a \emph{truth-table functional} (or $tt$-functional) if $\Phi$ is total.
\end{definition}

\begin{definition}
Given a $tt$-functional $\Phi:\cs \rightarrow \cs$, the \emph{measure induced by $\Phi$}, denoted $\lambda_\Phi$, is defined to be
\[\lambda_\Phi(\mathcal{X})=\lambda(\Phi^{-1}(\mathcal{X}))\]
for every measurable $\mathcal{X}\subseteq\cs$.  
\end{definition}

\noindent It is not hard to show that $\lambda_\Phi\in\M_c$ for each $tt$-functional $\Phi$.

Apart from the Lebesgue measure, the computable measures that we will consider here are all atomic, and even trivial.

\begin{definition}
Let $\mu\in\M_c$.
\begin{itemize}
\item[($i$)] $\mu$ is \emph{atomic} if there is some sequence $A\in\cs$ such that $\mu(\{A\})>0$. 
\item[($ii$)] $A\in\cs$ an \emph{atom} of $\mu$, denoted $A\in\atom_\mu$, if $\mu(\{A\})>0$. 
\item[($iii$)] $\mu$ is \emph{atomless} if $\atom_\mu=\emptyset$.
\item[($iv$)] $\mu$ is \emph{trivial} if $\mu(\atom_\mu)=1$.
\end{itemize}
\end{definition}

\noindent We will also refer to the atoms of $\mu$ as \emph{$\mu$-atoms}.

\subsection{Notions of algorithmic randomness}

In this section we introduce Martin-L\"of randomness (and relativizations thereof), Schnorr randomness, and weak 2-randomness.  

\begin{definition}
Let $\mu\in\M_c$. 
\begin{itemize}
\item[($i$)] A \emph{$\mu$-Martin-L\"of test} is a uniformly computable sequence $(\mathcal{U}_i)_{i\in\omega}$ of effectively open classes in $\cs$ such that $\mu(\mathcal{U}_i)\leq 2^{-i}$ for every $i\in\omega$.  
\item[($ii$)] $X\in\cs$ is \emph{$\mu$-Martin-L\"of random} if for every $\mu$-Martin-L\"of test $(\mathcal{U}_i)_{i\in\omega}$, we have $X\notin\bigcap_{i\in\omega}\mathcal{U}_i$. 
\end{itemize}
\end{definition}

The collection of $\mu$-Martin-L\"of random sequences will be written as $\MLR_\mu$ (we will simply write $\MLR$ when considering the Lebesgue measure).  The following fact, the existence of a universal Martin-L\"of test, is well-known and will prove to be useful here.

\begin{proposition}
For every $\mu\in\M_c$, there is a Martin-L\"of test $(\widehat{\mathcal{U}}_i)_{i\in\omega}$ such that $x\in\MLR_\mu$ if and only if $x\notin\bigcap_{i\in\omega}\widehat{\mathcal{U}}_i$.
\end{proposition}

We will make heavy use of the following lemma in Section \ref{sec:fdl}:

\begin{lemma}\label{lem:sum-meas}
Given $\mu,\nu\in\M_c$, if we set $\rho=\frac{\mu+\nu}{2}$,then 
\[
\MLR_\rho=\MLR_\mu\cup\MLR_\nu.\footnote{In general, we can consider any convex sum $\frac{\alpha\mu+(1-\alpha)\nu}{2}$ of $\mu$ and $\nu$, as long as $\alpha\in[0,1]$ is computable.} 
\]
\end{lemma}

\begin{proof}
If $X\notin\MLR_\rho$, then $X\in\bigcap_{i\in\omega}\mathcal{U}_i$, where $(\mathcal{U}_i)_{i\in\omega}$ is the universal $\rho$-Martin-L\"of test.  But since $\rho(\mathcal{U}_i)\leq 2^{-i}$ for each $i\in\omega$, it follows that $\mu(\mathcal{U}_{i+1})\leq 2^{-i}$ and $\nu(\mathcal{U}_{i+1})\leq 2^{-i}$, and hence it follows that $X\notin\MLR_\mu\cup\MLR_\nu$.

Conversely, if $X\notin\MLR_\mu\cup\MLR_\nu$, then $X\in\bigcap_{i\in\omega}\mathcal{U}_i$ and $X\in\bigcap_{i\in\omega}\mathcal{V}_i$, where $(\mathcal{U}_i)_{i\in\omega}$ is the universal $\mu$-Martin-L\"of test and $(\mathcal{V}_i)_{i\in\omega}$ is the universal $\nu$-Martin-L\"of test.  Then since $\mathcal{U}_i\cap\mathcal{V}_i\subseteq\mathcal{U}_i$ and $\mathcal{U}_i\cap\mathcal{V}_i\subseteq\mathcal{V}_i$, it follows that $(\mathcal{U}_i\cap\mathcal{V}_i)_{i\in\omega}$ is a $\rho$-Martin-L\"of test, and hence $X\notin\MLR_\rho$.
\end{proof}

We will also draw heavily on the fact that there are $\Delta^0_2$ Martin-L\"of random sequences.  Recall that a sequence $A\in\cs$ is $\Delta^0_2$ if there is a uniformly computable sequence of finite sets $(A_s)_{s\in\omega}$ (called a $\Delta^0_2$ approximation of $A$) such that
\[
\lim_{n\rightarrow\infty}A_s(n)=A(n)
\]
for every $n\in\omega$. For example, Chaitin's $\Omega$, defined by
\[
\Omega:=\sum_{U(\sigma)\halts}2^{-|\sigma|},
\]
where $U$ is a universal prefix-free Turing machine, is a $\Delta^0_2$ Martin-L\"of random sequence.  In fact, $\Omega\equiv_T\emptyset'$.

The following results concerning $\mu$-atoms and their relationship to $\mu$-Martin-L\"of randomness will be particularly useful.

\begin{proposition}\label{prop:atoms}
Let $\mu\in\M_c$ and $X\in\cs$.
\begin{itemize}
\item[($i$)] If $X\in\atom_\mu$, then $X$ is computable.
\item[($ii$)] If $X\in\atom_\mu$, then $X\in\MLR_\mu$.
\item[($iii$)] If $X$ is computable and $\mu(\{X\})=0$, then $X\notin\MLR_\mu$.
\end{itemize}
\end{proposition}

\noindent Part ($ii$) of the above proposition actually holds for all notions of randomness that we consider here, since for each such definition, the non-random sequences are captured in some set of $\mu$-measure zero.

We will also consider relativized versions of Martin-L\"of randomness.  For $A\in\cs$, a $\mu$-Martin-L\"of test relative to $A$ is simply a uniformly $A$-computable sequence $(\U_i^A)_{i\in\omega}$ of $\Sigma^{0,A}_1$ classes in $\cs$ such that $\mu(\U_i^A)\leq 2^{-i}$ for every $i\in\omega$.  Moreover, $X\in\cs$ is $\mu$-Martin-L\"of random relative to $A$, denoted $X\in\MLR_\mu^A$, if $X\notin\bigcap_{i\in\omega}\U_i^A$ for any $\mu$-Martin-L\"of test relative to $A$ $(\U_i^A)_{i\in\omega}$.

If we relativize Martin-L\"of randomness to the halting set $\emptyset'$, then this gives rise to 2-randomness.  We set $\twoMLR:=\MLR^{\emptyset'}$.  It is immediate that $X\in\twoMLR$ if and only if $X\in\MLR^A$ for some $A\equiv_T\emptyset'$.  It is not hard to show that $\twoMLR\subseteq\WtwoR$.

One of the central results concerning relative randomness is known as ``van Lambalgen's theorem."

\begin{theorem}[\cite{Lam90}]\label{thm:vlt}
For every $A,B\in\cs$, 
\[
A\oplus B\in\MLR\;\text{if and only if}\;A\in\MLR^B\;\&\;B\in\MLR.
\]
\end{theorem}  

\noindent A related result is the following.

\begin{theorem}[\cite{Kau91}]\label{thm:bigoplus}
Given $A\in\MLR$, then for each $i\in\omega$, $A^{[i]}\in\MLR^{\bigoplus_{j\neq i} A^{[j]}}$.  Moreover, for any finite $J\subseteq\omega$, $A^{[i]}\in\MLR^{\bigoplus_{j\in J}A^{[j]}}$ for every $i\notin J$.
\end{theorem}

Next we turn to Schnorr randomness, the definition of which is a slight variant of the definition of Martin-L\"of randomness.

\begin{definition}
Let $\mu\in\M_c$. 
\begin{itemize}
\item[($i$)] A \emph{$\mu$-Schnorr test} is a uniformly computable sequence $(\mathcal{U}_i)_{i\in\omega}$ of effectively open classes in $\cs$ such that $\mu(\mathcal{U}_i)=2^{-i}$ for every $i\in\omega$.  
\item[($ii$)] $X\in\cs$ is \emph{$\mu$-Schnorr random} if for every $\mu$-Schnorr test $(\mathcal{U}_i)_{i\in\omega}$, we have $X\notin\bigcap_{i\in\omega}\mathcal{U}_i$. 
\end{itemize}
\end{definition}

Let $\SR_\mu$ be the class of $\mu$-Schnorr random sequences.  One can readily observe that $\MLR_\mu\subseteq\SR_\mu$ for each $\mu\in\M_c$.  
In order to separate Martin-L\"of randomness and Schnorr randomness with respect to a trivial computable measure $\mu$, we need the following result. 

\begin{proposition}[Nies, Stephan, Terwijn \cite{NieSteTer05}]\label{prop:sr-high}
Given $\mu\in\M_c$, every $X\in\SR_\mu\setminus\MLR_\mu$ is high, i.e. $X$ computes a function that dominates all computable functions.
\end{proposition}

We will review the proof of this proposition, as the details will be useful when we provide a counterexample to Schnorr's claim in Section \ref{sec:separating}.

\begin{proof}
Let $X\in\SR_\mu\setminus\MLR_\mu$, and let $(\mathcal{U}_i)_{i\in\omega}$ be a universal $\mu$-Martin-L\"of test.  Then we define $f\leq_T X$ as follows:
\[
f(n)=\;\text{the least $s$ such that}\;(\exists k)\llb X\uh k\rrb\subseteq\mathcal{U}_{n,s}.
\]
We claim that $f$ dominates all computable functions.  Suppose, for the sake of contradiction, that there is some computable function $g$ such that $f(n)\leq g(n)$ for infinitely many $n$.  Then $(\mathcal{U}_{n,g(n)})_{n\in\omega}$ is in fact a Schnorr test, and moreover, there are infinitely many $n$ such that $X\in\mathcal{U}_{n,g(n)}$, which is, in fact, sufficient to show that $X$ is not Schnorr random (for instance, see \cite[Theorem 7.1.10]{DowHir10}).  Thus, $X$ computes a function that dominates all computable functions, and hence $X$ is high.
\end{proof}

Another definition of randomness that we consider here is weak 2-randomness, first introduced by Kurtz in his dissertation \cite{Kur81}.

\begin{definition}
Let $\mu\in\M_c$. 
\begin{itemize}
\item[($i$)] A \emph{generalized $\mu$-Martin-L\"of test} is a uniformly computable sequence $(\mathcal{U}_i)_{i\in\omega}$ of effectively open classes in $\cs$ such that $\lim_{i\rightarrow\infty}\mu(\mathcal{U}_i)=0$.
\item[($ii$)] $X\in\cs$ is \emph{$\mu$-weakly 2-random} if for every generalized $\mu$-Martin-L\"of test $(\mathcal{U}_i)_{i\in\omega}$, we have $X\notin\bigcap_{i\in\omega}\mathcal{U}_i$. 
\end{itemize}
\end{definition}

Let $\WtwoR_\mu$ denote the collection of $\mu$-weakly 2-random sequences.  Note that every generalized $\mu$-Martin-L\"of test $(\U_i)_{i\in\omega}$ yields a $\Pi^0_2$ class of $\mu$-measure zero, namely $\bigcap_i\U_i$, and conversely, every $\Pi^0_2$ class of $\mu$-measure zero can be obtained in this way by a generalized $\mu$-Martin-L\"of test.  It is immediate that $\WtwoR_\mu\subseteq\MLR_\mu$, since the collection of sequences captured by a $\mu$-Martin-L\"of defines a $\Pi^0_2$ class of $\mu$-measure zero.  We now consider the key result concerning the relationship between $\WtwoR_\mu$ and $\MLR_\mu$.

\begin{definition}
$X,Y\in\cs$ form a \emph{minimal pair} in the Turing degrees if $A<_T X$ and $A<_T Y$ implies that $A\equiv_T\emptyset$.
\end{definition}

\begin{theorem}[Downey, Nies, Weber, Yu \cite{DowNieWeb06}; Hirschfeldt, Miller (unpublished)]\label{thm:w2r-minpair}
For $\mu\in\M_c$, if $X$ is not computable, then $X\in\WtwoR_\mu$ if and only if $X\in\MLR_\mu$ and $X$ and $\emptyset'$ form a minimal pair.
\end{theorem}

The proof of this result as found in the literature is only for the case of the Lebesgue measure, but it is straightforward to extend it to any $\mu\in\M_c$. We follow the original proof, with several modifications.

\begin{theorem}[Downey, Nies, Weber, and Yu \cite{DowNieWeb06}]\label{thm-w2r-minpair-1}
For $\mu\in\M_c$, if $X\in\WtwoR_\mu$ and $X$ is not computable, then $X$ and $\emptyset'$ form a minimal pair.
\end{theorem}

\begin{proof}[Sketch]
We modify the proof given by Downey, Nies, Weber, and Yu for the case that $\mu=\lambda$.  If $A\in\cs$ is $\Delta^0_2$, $Z\in\WtwoR_\mu$, and $\Phi^Z=A$ for some Turing functional $A$,  then we argue that $A$ is computable.  Towards this end, we define
\[
\S=\{X:\forall n\forall s\exists t>s(\Phi^X(n)[t]\halts=A_t(n)\},
\]
which is $\Pi^0_2$ and contains $Z$.  Since $Z\in\WtwoR_\mu$, it follows that $\mu(\S)>0$.  Now the only difference between the original proof and the situation here is that $\mu$ may be atomic, so that there is some $\mu$-atom $Y\in\S$.  But in this case we are done: since $\mu$ is computable, it follows that $Y$ is computable.  Then $\Phi^Y=A$, and hence $A$ is computable.

In the case that $\S$ contains no atoms, the proof proceeds exactly as in the case of the Lebesgue measure:  by a ``majority vote" argument, which shows that that one can compute values of $A$ using the majority of sequences in a set of positive measure, one shows that $A$ is computable.  See, for instance, the proof of Theorem 7.2.8. in \cite{DowHir10} for details.
\end{proof}

\begin{theorem}\label{thm:w2r-minpair-2}[Hirschfeldt, Miller (unpublished)]
Let $\mu\in\M_c$.  For any $\Sigma^0_3$ class $\S\subseteq\cs$ such that $\mu(\S)=0$, there is a noncomputable c.e.\ set $A$ such that $A\leq_T X$ for every non-computable $X\in\MLR_\mu\cap\S$.
\end{theorem}

The proof of this result proceeds exactly in the same way as the case of the Lebesgue measure, since $X\in\MLR_\mu\cap\S$ implies that $\mu(\{X\})=0$ and hence $X$ is not a $\mu$-atom.  See the proof of Theorem 7.2.11 of \cite{DowHir10}.\\

\begin{proof}[Proof of Theorem \ref{thm:w2r-minpair}]
($\Rightarrow$) This is simply Theorem \ref{thm-w2r-minpair-1}.

\noindent
($\Leftarrow$) Suppose that $X\in\MLR_\mu\setminus\WtwoR_\mu$.  Then there is a $\Pi^0_2$ $\mu$-null set $\S$ such that $X\in\S$, and so by Theorem \ref{thm:w2r-minpair-2}, there is some non-computable c.e.\ set $A$ such that $A\leq_T X$.  Therefore, $X$ and $\emptyset'$ do not form a minimal pair.
\end{proof}

\subsection{Randomness and Turing functionals}

Many of the results in the sequel depend crucially on the following preservation of randomness theorem, originally due to Levin and Zvonkin \cite{ZvoLev70}.

\begin{theorem}[Preservation of Martin-L\"of Randomness]\label{thm:preservation-mlr}
If $\Phi$ is a $tt$-functional, then $X\in\MLR$ implies $\Phi(X)\in\MLR_{\lambda_\Phi}$.
\end{theorem}

\noindent We will also use the following variant established by Bienvenu and Porter \cite{BiePor12}.

\begin{theorem}[Preservation of Schnorr Randomness]\label{thm:preservation-sr}
If $\Phi$ is an $tt$-functional, then $X\in\SR$ implies $\Phi(X)\in\SR_{\lambda_\Phi}$.
\end{theorem}

\noindent The basic idea behind these two proofs is straightforward.  Given a Martin-L\"of test $(\U_i)_{i\in\omega}$ (resp.\ Schnorr test) with respect to the measure $\lambda_\Phi$ induced by $\Phi$, one shows that the pre-image of $(\U_i)_{i\in\omega}$ under $\Phi$ is a Martin-L\"of test (resp.\ Schnorr test) with respect to the Lebesgue measure.  Thus, if $\Phi(X)$ is captured by the test $(\U_i)_{i\in\omega}$, $X$ is covered by the pre-image of this test under $\Phi$.  For more details, see, for instance, \cite[Theorem 3.2 and Theorem 4.1]{BiePor12}.

We will also make use of a relative version of the preservation of Martin-L\"of randomness:  For any $A\in\cs$, if $\Phi$ is a $tt$-functional and $X\in\MLR^A$, then $\Phi(X)\in\MLR_{\lambda_\Phi}^A$.  

A partial converse to the preservation of Martin-L\"of randomness due to Shen (unpublished) is the following.

\begin{theorem}\label{thm:ex-nihilo}
For $\Phi$ an almost total functional and $Y\in\cs$, if $Y\in\MLR_{\lambda_\Phi}$, then there is some $X\in\MLR$ such that $\Phi(X)=Y$.
\end{theorem}

A proof of this result can be found in \cite[Theorem 3.5]{BiePor12}.  A relativized version of Theorem \ref{thm:ex-nihilo} also holds, and combining this result with the relativized version of the preservation of Martin-L\"of randomness, one can prove the following:

\begin{theorem}\label{thm:rel-lk}
Given $X, A\in\cs$ and a $tt$-functional $\Phi$, if $\Phi(X)$ is not computable and $\Phi^{-1}(\Phi(X))=\{X\}$, then $X\in\MLR^A$ if and only if $\Phi(X)\in\MLR_{\lambda_\Phi}^A$.
\end{theorem}

\section{Tally Functionals}\label{sec:tally}

The main tool used in the construction of various trivial measures are what we refer to here as \emph{tally functionals}.  Given a $\Delta^0_1$ formula $\Theta(X,y,z)$ with free first-order variables $y$ and $z$ and free second-order variable $X$, we define an auxiliary function $\theta(A,n):\cs\times\omega\rightarrow\omega$ by
\[
\theta(A,n)=
\left\{
	\begin{array}{ll}
		 \;\text{the least $s$ such that}\; \Theta(A,n,s) & \mbox{if   $s$ exists}  \\
		+\infty & \mbox{otherwise} 
	\end{array}.
\right.
\]
Then the tally functional $\Phi_\Theta$ determined by the formula $\Theta$ is defined to be
\[
\Phi_\Theta(A)=1^{\theta(A,0)}\;0\;1^{\theta(A,1)}\;0\;1^{\theta(A,2)}\;0\dotsc,
\]
where 
\[\Phi_\Theta(A)=1^{\theta(A,0)}\;0\;1^{\theta(A,1)}\;0\dotsc\;1^{\theta(A,k)}\;0\;1^\omega\]
if $\theta(A,k+1)=+\infty$ (and $k+1$ is the least $n$ such that $\theta(A,n)=+\infty$).

A useful example that we will use repeatedly in Section \ref{sec:fdl} is a tally functional defined in terms of the approximation of a non-computable $\Delta^0_2$ sequence $A$.  Let $(A_s)_{s\in\omega}$ be a $\Delta^0_2$ approximation of $A$.  Without loss of generality, we can assume that $A_s\neq A_{s+1}$ for every $s$.  If we define the formula $\Theta(A,n,s)$ so that
\[
\Theta(A,n,s)\;\text{holds if and only if}\;A\uh n=A_s\uh n,
\] 
then $\theta(A,n)$ is the first stage $s$ such that $A\uh n=A_s\uh n$.  Let $\Phi_A$ be the resulting tally functional.

\begin{lemma}\label{lem:tally-A}
Suppose that $A\in\Delta^0_2$ is non-computable, and let $\Phi_A$ be the tally functional defined above.
\begin{itemize}
\item[($i$)] If $X=A_s$ for some $s$, then there is some $m$ such that for every $n\geq m$, $\theta(X,n)=\theta(X,m)$, i.e., the function $f(x)=\theta(X,x)$ is eventually constant. 
\item[($ii$)] If $X\neq A$ and $X\neq A_s$ for each $s$, then $\theta(X,n)=+\infty$ for some $n$. 
\item[($iii$)] The function $g(x)=\theta(A,n)$ is not computable, nor is it dominated by any computable function.  
\end{itemize}
\end{lemma}

\begin{proof} ($i$)  Let $s$ be least such that $X=A_s$.  If $s=0$, then since for all $n$, $X\uh n=A_0\uh n$, it follows that $\theta(X,n)=0$ for all $n$.  Now suppose that $s>0$.  Then there is some $m$ such that $X\uh m \neq A_{s-1}\uh m$.  It thus follows that $X \uh n=A_s\uh n$ for all $n\geq m$, which implies that $\theta(X,n)=s$ for all $n\geq m$.

\medskip

\noindent($ii$) First we establish the following claim:  There is some $k$ such that for every $s$, $X\uh k\neq A_s\uh k$.  Suppose not, so that for every $k$, there is some $s$ such that $X\uh k=A_s\uh k$.  But this implies that for every $k$, there exist infinitely many $s$ such that $X\uh k=A_s\uh k$, since (a) $X\neq A_s$ for every $s$ and (b) $X\uh k =A_s\uh k$ implies that $X\uh j =A_s\uh j$  for every $j<k$.  From this it follows that the $A_s$'s (viewed as rational numbers) converge to $X$ (viewed as a real number).  But the $A_s$'s also converge to $A$, and thus it follows that $X=A$, contradicting our hypothesis.  Now let $k$ be least such $X\uh k\neq A_s\uh k$ for every $s$.  Then it follows that $\theta(X,k)=+\infty$.

\medskip

\noindent($iii$)  Suppose that $g(n)=\theta(A,n)$ is computable.  Then since 
\[
A_{g(n)}\uh n=A_{\theta(A,n)}\uh n=A\uh n,
\]
this implies that $A$ is computable, contradicting our assumption that $A$ is non-computable.  The proof that $g$ is not dominated by any computable function is just given by the standard proof used to show that every non-computable $\Delta^0_2$ sequence is hyperimmune.  See, for example, \cite[Theorem 1.5.12]{Nie09}.

\end{proof}

\begin{theorem}\label{thm:delta2}
\begin{itemize}
\item[($i$)] The measure $\mu$ induced by $\Phi_A$ is trivial.
\item[($ii$)] If $A\in\MLR$, then $\MLR_\mu\setminus\atom_\mu=\{\Phi_A(A)\}$.
\end{itemize}
\end{theorem}

\begin{proof}
($i$) Given input $X$, there are three cases to consider to determine the output $\Phi_A(X)$:\\

\noindent\emph{Case 1:} $X=A_s$ for some $s$.  In this case, 
	by Lemma \ref{lem:tally-A}($i$), the function $f(x)=\theta(X,x)$ is eventually constant, and thus $\Phi_A(X)=\sigma1^k01^k01^k0\dotsc$ for some $\sigma\in\str$, which is clearly computable.\\

\noindent\emph{Case 2:}  $X\neq A$ and $X\neq A_s$ for every 	$s$.  Then by Lemma \ref{lem:tally-A}($ii$), there is some 	$n$ such that $\theta(X,n)=+\infty$, and hence $\Phi_A(X)=\sigma1^\omega$ for some $\sigma\in\str$, which is computable.\\

\noindent\emph{Case 3:}  $X=A$.  By Lemma \ref{lem:tally-A}($iii$), since the function $g(n)=\theta(A,n)$ is not computable, it follows that 
	\[
	\Phi_A(X)=\Phi_A(A)=1^{\theta(A,0)}\;0\;1^{\theta(A,1)}\;0\;1^{\theta(A,2)}\;0\dotsc
	\]
	is not a computable sequence.

For every $B\in\MLR$ such that $B\neq A$, we must be in Case 2, as $B\neq A_s$ for every $s$.  Thus $\Phi_\Theta(B)=\sigma 1^\omega$ for some $\sigma\in\str$. Setting
\[
\mathcal{S}=\{Y:(\exists \sigma\in\str)[Y=\sigma1^\omega]\},
\]
we have
\[
\MLR\setminus\{A\}\subseteq\Phi^{-1}_\Theta(\mathcal{S}),
\]
from which it follows that 
\[
1=\lambda(\MLR\setminus\{A\})\leq \lambda(\Phi^{-1}_\Theta(\mathcal{S})).
\]
Since $\mu=\lambda_{\Phi_\Theta}$ assigns measure one to the countable collection $\mathcal{S}$, it follows that $\mu$ is trivial.

\bigskip
\noindent($ii$)  First, $\Phi_A(A)\in\MLR_\mu$ by the preservation of Martin-L\"of randomness.  In addition, $\Phi_A(A)\notin\atom_\mu$, for otherwise $\Phi_A(A)$ would be computable by Proposition \ref{prop:atoms}($i$), and hence the function $g(n)=\theta(A,n)$ would be computable, contradicting Lemma \ref{lem:tally-A}($iii$).

Next, if $X\in\MLR_\mu\setminus\atom_\mu$, then by Theorem \ref{thm:ex-nihilo}, $X\in\MLR_\mu$ implies that $X=\Phi_A(Y)$ for some $Y\in\MLR$.  But since $X\notin\atom_\mu$, $X$ is not computable.  In particular, $X$ does not have the form $\sigma 1^\omega$ for any $\sigma\in\str$. It follows that $Y$ cannot fall under Case 2, and since no $Y\in\MLR$ falls under Case 1, it must be that $Y=A$.  Thus, $X=\Phi_A(A)$.

\end{proof}

We will revisit this result in Section \ref{sec:fdl}, when we consider the $LR$-degree structures associated to different trivial measures. For other applications of tally functionals, see \cite{BiePor12}.

\section{Separating Randomness Notions via Trivial Measures}\label{sec:separating}

In this section, we prove three results of the following form:  Let $\R^1$ and $\R^2$ be two notions of randomness such that $\R^1_\mu\subseteq\R^2_\mu$ for every $\mu\in\M_c$.  Then there is a measure $\mu\in\M_c$ such that ($i$) $\R^1_\mu=\atom_\mu$ and ($ii$) $\R^2_\mu\setminus\R^1_\mu\neq\emptyset$.  Moreover, we will be able to conclude that $\mu$ is trivial, since by condition ($i$), $\mu(\atom_\mu)=1$.

\subsection{Separating Martin-L\"of randomness and 2-randomness}  In order to separate $\MLR$ and $\twoMLR$ via a trivial measure, we prove a more general result by modifying the construction of a trivial measure in the previous section, and then we apply van Lambalgen's Theorem.

\begin{theorem}\label{thm:mlr-A}
Let $A\in\MLR\cap\Delta^0_2$.  Then there is a trivial measure $\mu\in\M_c$ such that
\begin{itemize}
\item[($i$)] $\MLR_\mu^A=\atom_\mu$, and
\item[($ii$)] $\MLR_\mu\setminus\MLR_\mu^A\neq\emptyset$.
\end{itemize}
\end{theorem}

\begin{proof}
We define a new functional $\Psi$ that on input $X\oplus Y$ behaves much like the tally functional $\Phi_A$ defined in Section \ref{sec:tally}. However, instead having our functional output a sequence of blocks of 1s separated by individual 0s, $\Psi$ will output blocks consisting the bits of $Y$.  Specifically, suppose that 
\[
\Phi_A(X)=1^{t_0}01^{t_1}01^{t_2}0\dotsc1^{t_i}0\dotsc.
\]
Then we have 
\[
\Psi(X\oplus Y)=y_0^{t_0}\;y_1^{t_1}\;y_2^{t_2} \dotsc y_i^{t_i}\dotsc
\]
where $y_i=Y(i)$ for every $i$.  Note that $\Psi$ is total, since $\Phi_A$ is total.  Let $\mu$ be the induced measure $\lambda_\Psi$.  Given input $X\oplus Y$, there are three relevant cases to consider, corresponding to the three cases we considered in the proof of Theorem \ref{thm:delta2}.\\

\noindent\emph{Case 1:} $X=A_s$ for some $s$.  Then for any $Y\in\cs$, there is some $k$ such that
\[
\Psi(A_s\oplus Y)=y_0^{t_0}\;y_1^{t_1}\;y_2^{t_2} \dotsc y_i^k\; y_{i+1}^k\; y_{i+2}^k\dotsc .
\]
That is, the lengths of the blocks of the $y_i$'s eventually stabilize (just as the function $\theta(A_s,n)$ in Case 1 of the proof of Theorem \ref{thm:delta2} eventually stabilizes).\\

\noindent\emph{Case 2:}  $X\neq A$ and $X\neq A_s$ for every 	$s$.  Then there is some $\sigma\in\str$ and $i\in\omega$ such that after some stage, $\Psi$ will output the same bit $y_i$ forever, i.e.,
\[
\Psi(X\oplus Y)=\sigma(y_i)^\omega.
\]

\noindent\emph{Case 3:}  $X=A$.  Then using the same function $g(n)=\theta(A,n)$ from Lemma \ref{lem:tally-A}($iii$),  we have	
\[
	\Psi(X\oplus Y)=\Psi(A\oplus Y)=y_1^{\theta(A,0)}\;y_2^{\theta(A,1)}\;y_3^{\theta(A,2)}\dotsc.
	\]
Note further that if $Y$ has only finitely many 0s or finitely many 1s, then $\Psi(A\oplus Y)$ will eventually stabilize.  In the case that $Y$ has infinitely 0s and 1s, there is some function $f:\omega\rightarrow\omega$ such that
\[
\Psi(A\oplus Y)=b_0^{f(0)}\;b_1^{f(1)}\;b_2^{f(2)} \dotsc y_i^{f(i)}\dotsc
\]
where $b_i\neq b_{i+1}$ for each $i\in\omega$.  It is not hard to see that $\theta(A,n)\leq f(n)$ for every $n\in\omega$, from which it follows that $f$ is not computable by Lemma \ref{lem:tally-A}($iii$), and thus $\Psi(A\oplus Y)$ is not computable.

Now we verify ($i$) and ($ii$).  For ($i$), it is clear that $\atom_\mu\subseteq\MLR_\mu^A$. Given $Z\in\MLR_\mu^A$, by Theorem \ref{thm:ex-nihilo} relativized to $A$, there must be some $A$-random sequence $X\oplus Y$ such that $\Psi(X\oplus Y)=Z$.  If $Z$ is not computable, it must be the case that either $X=A$ or $X=A_s$ for some $s\in\omega$, for otherwise we would be in Case 2, so that $\Psi(X\oplus Y)=Z$ would be computable.  But $A\oplus Y$ is not Martin-L\"of random relative to $A$, nor is $A_s\oplus Y$ for any $s\in\omega$.  Thus, $Z$ must be computable, and so by Proposition \ref{prop:atoms}($iii$), $Z\in\atom_\mu$.

For ($ii$), let $B\in\MLR^A$.  Then $A\oplus B\in\MLR$ by van Lambalgen's Theorem (Theorem \ref{thm:vlt}).  It follows from the preservation of Martin-L\"of randomness that $\Psi(A\oplus B)\in\MLR_\mu$.  Since $B$ has infinitely many 0s and 1s, by the discussion in Case 3, $\Psi(A\oplus B)$ is not computable.  By Proposition \ref{prop:atoms}($i$), $\Psi(A\oplus B)\notin\atom_\mu$, and so by part ($i$), $\Psi(A\oplus B)\notin\MLR_\mu^A$. 
\end{proof}

We now have the following corollary:
\begin{corollary}\label{cor:2mlr}
There is a trivial measure $\mu$ such that
\begin{itemize}
\item[($i$)] $\twoMLR_\mu=\atom_\mu$, and
\item[($ii$)] $\MLR_\mu\setminus\twoMLR_\mu\neq\emptyset$.
\end{itemize}
\end{corollary}

\begin{proof}
Choose any $A\in\Delta^0_2\cap\MLR$ such that $A\equiv_T\emptyset'$ (for instance, let $A=\Omega$) and apply Theorem \ref{thm:mlr-A}.
\end{proof}

\subsection{Separating Martin-L\"of randomness and weak 2-randomness}

We can improve Corollary \ref{cor:2mlr} by replacing $\twoMLR$ with $\WtwoR$.  However, we need to use a different technique to do so:  we will use the characterization provided by Theorem \ref{thm:w2r-minpair}, that for each non-computable $X$, $X\in\WtwoR_\mu$ if and only $X\in\MLR_\mu$ and $X$ forms a minimal pair with $\emptyset'$.

\begin{theorem}
There is a trivial measure $\mu\in\M_c$ such that
\begin{itemize}
\item[($i$)] $\WtwoR_\mu=\atom_\mu$, and
\item[($ii$)] $\MLR_\mu\setminus\WtwoR_\mu\neq\emptyset$.
\end{itemize}
\end{theorem}

\begin{proof}
Let $A\in\MLR\setminus\WtwoR$.  Then by Theorem \ref{thm:w2r-minpair} there is some Turing functional $\Gamma$ and a non-computable $\Delta^0_2$ sequence $B$ such that $\Gamma(A)=B$.  Let $\Theta(X,n,s)$ be such that
\[
\Theta(X,n,s)\;\text{holds if and only if}\;(\forall k<n)\Gamma_s(X)(k)\halts=B_s(k),
\] 
where $(B_s)_{s\in\omega}$ is some fixed $\Delta^0_2$ approximation of $B$. Without loss of generality, $B_s\neq B_{s+1}$ for every $s\in\omega$.  Then it follows that $\theta$ is defined to be
\[
\theta(X,n)=\twopartdefow{\text{the least } s \text{ such that }(\forall k<n)\Gamma_s(X)(k)\halts=B_s(k)}{s\;\text{exists}}{+\infty}{}.
\]
Given $X\in\cs$, to compute $\Phi_\Theta(X)$, we have four cases to consider, each depending on the behavior of $\Gamma$ on input $X$.\\  

\noindent\emph{Case 1:} $\Gamma(X)\diverge$.  Then there is some least $n$ such that $\Gamma(X)(n)\diverge$, and so $\theta(X,k)=+\infty$ for some $k\leq n+1$ (there may be some $k<n+1$ such that $\Gamma(X)(k)\halts$ but $\Gamma(X)(k)\neq B_s(k)$ for every $s\in\omega$).  Consequently,
\begin{equation}\label{eq:w2r-triv1}
\Phi_\Theta(X)=\sigma1^\omega
\end{equation}
for some $\sigma\in\str$.\\ 


\noindent\emph{Case 2:}  $\Gamma(X)\halts=B_t$ for some $t\in\omega$.  Since $B_s\neq B_{s+1}$ for every $s\in\omega$, there is some least $n$ such that for all sufficiently large $s\geq n$,
\[
(\forall k<n)[\Gamma_s(k)\halts=B_t(k)]
\]
but
\[
(\exists j<n)[B_s(j)\neq B_t(j)].
\]
That is, the functional $\Phi_\Theta$ on input $X$ waits to see $\Gamma_s(X)$ agree with $B_s$, but for sufficiently large stages $s$, $\Gamma_s(X)$ only agrees with $B_t$.  Hence for the $n$ given above, we have $\theta(X,n)=+\infty$, and thus  (\ref{eq:w2r-triv1}) holds. \\

\noindent\emph{Case 3:}  $\Gamma(X)\halts$ but $\Gamma(X)\neq B$ and $\Gamma(X)\neq B_s$ for every $s\in\omega$.  In this case we have
\begin{equation}\label{eq:w2r-triv2}
(\exists n)(\forall s)(\exists k<n)[\Gamma_s(X)(k)\neq B_s(k)].
\end{equation}
The argument to establish this claim is the same as the one we gave in the proof of Lemma \ref{lem:tally-A}($ii$).  Let $n$ be the least number satisfying  (\ref{eq:w2r-triv2}). Then we have $\theta(X,n)=+\infty$, and thus (\ref{eq:w2r-triv1}) holds. \\

\noindent\emph{Case 4:}  $\Gamma(X)\halts=B$.  Setting $h(n):=\theta(X,n)$, we have
\[
\Phi_\Theta(X)=1^{h(0)}01^{h(1)}0\dotsc,
\]
so that $\Phi_\Theta(X)\geq_T h$. But for every  $n\in\omega$, 
\[
B_{h(n)}\uh n=\Gamma_{h(n)}(X)\uh n=\Gamma(X)\uh n=B\uh n,
\] 
so it follows that $\Phi_\Theta(X)\geq_T B$.\\

To establish ($i$), $\atom_\mu\subseteq\WtwoR_\mu$ is immediate.  For the other direction, consider $Y\in\WtwoR_\mu$, which by Theorem \ref{thm:w2r-minpair} forms a minimal pair with $\emptyset'$.  Since $Y\in\MLR_\mu$, by Theorem \ref{thm:ex-nihilo} there is some $X\in\MLR$ such that $\Phi_\Theta(X)=Y$.  

Applying the functional $\Gamma$ to $X$ yields one of two general outcomes:  either one of Case 1, 2, or 3 occurs, in which case  
\[
Y=\Phi_\Theta(X)=\sigma1^\omega
\]
and hence $Y\in\atom_\mu$ by Proposition \ref{prop:atoms}($ii$), or we are in Case 4.  But in this case, $Y=\Phi_\Theta(X)\geq_T B$, contradicting the fact that $Y$ forms a minimal pair with $\emptyset'$.  So we must have $Y\in\atom_\mu$.

For part ($ii$), if we take the original $A\in\MLR\setminus\WtwoR$ such that $\Gamma(A)=B$ that we started with, by the preservation of Martin-L\"of randomness, $\Phi_\Theta(A)\in\MLR_\mu$.  In addition, by Case 4 above, $\Phi_\Theta(A)\geq_T B$, so $\Phi_\Theta(A)$ does not form a minimal pair with $\emptyset'$.  Hence $\Phi_\Theta(A)\notin\WtwoR_\mu$.

\end{proof}

\subsection{Separating Martin-L\"of randomness and Schnorr randomness}

We end this section by using a tally functional to show that Schnorr's claim, namely that $\MLR_\mu=\SR_\mu$ for a computable measure $\mu$ if and only if $\mu$ is trivial, is false.  Here we will use yet another technique for constructing the requisite tally functional.

\begin{theorem}
There is a trivial measure $\mu\in\M_c$ such that 
\begin{itemize}
\item[($i$)] $\MLR_\mu=\atom_\mu$, and
\item[($ii$)] $\SR_\mu\setminus\MLR_\mu\neq\emptyset$.
\end{itemize}
\end{theorem}

\begin{proof}
To construct the desired measure $\mu$, we define a tally functional in terms of a universal $\mu$-Martin-L\"of test $(\U_i)_{i\in\omega}$.  Let $\Theta(A,n,s)$ be such that 
\[
\Theta(A,n,s)\;\text{if and only if}\;(\exists k<s)\llb A\uh k\rrb\subseteq\mathcal{U}_{n,s}.
\]
As above, $\theta(A,n)$ is the least such $s$ such that $\Theta(A,n,s)$ holds, or is $+\infty$ if $\Theta(A,n,s)$ fails to hold for every $s$.

There are two cases of interest to us (the case that $X\not\in\SR$ has no bearing on the result here).\\

\noindent\emph{Case 1:  }$X\in\MLR$.  In this case, there is some least $n$ such that $X\notin\mathcal{U}_n$, and hence $\theta(X,n)=+\infty$, so that
\[
\Phi_\Theta(X)=\sigma 1^\omega
\]
for some $\sigma\in\str$.\\

\noindent\emph{Case 2:  }$X\in\SR\setminus\MLR$. Then $X\in\bigcap_{i\in\omega}\mathcal{U}_i$, and moreover, by the proof of Theorem \ref{prop:sr-high}, the function $f\leq_T X$ such that
\[
f(n)=\;\text{the least $s$ such that}\;(\exists k)\llb X\uh k\rrb\subseteq\mathcal{U}_{n,s}.
\]
 dominates all computable functions.  It follows that $\theta(X,n)=f(n)$, so that
\[
\Phi_\Theta(X)=1^{f(0)}\;0\;1^{f(1)}\;0\;\dotsc
\]
is not computable.\\

To verify ($i$), as above $\atom_\mu\subseteq\MLR_\mu$ is immediate.  Now given $Z\in\MLR_\mu$, by Theorem \ref{thm:ex-nihilo}, there is some $X\in\MLR$ such that $\Phi_\Theta(X)=Z$.  But we are in Case 1, so $Z=\sigma1^\omega$ for some $\sigma\in\str$.  By Proposition \ref{prop:atoms}($iii$), $Z\in\atom_\mu$.

For ($ii$), given any $X\in\SR\setminus\MLR$, by the preservation of Schnorr randomness, we have $\Phi_\Theta(X)\in\SR_\mu$.  Since we are in Case 2, $\Phi_\Theta(X)$ is not computable, and thus by Proposition \ref{prop:atoms}($i$) $\Phi_\Theta(X)\notin\atom_\mu$.  Thus by part ($i$), $\Phi_\Theta(X)\notin\MLR_\mu$.

\end{proof}

We have shown that one direction of Schnorr's claim is false, namely that $\mu$ being trivial does not imply that $\MLR_\mu=\SR_\mu$.  However, the status of the other direction of Schnorr's claim is also unclear.

\begin{question}
For $\mu\in\M_c$, does $\MLR_\mu=\SR_\mu$ imply that $\mu$ is trivial?
\end{question}

\section{Trivial Measures and Finite Distributive Lattices}\label{sec:fdl}

As further evidence of the non-triviality of trivial measures, we show that each trivial measure gives rise to a certain degree structure.  Specifically, if one considers the $LR$-degrees (or ``low-for-random" degrees) associated with $\MLR_\mu$ for a trivial measure $\mu$, one finds that different trivial measures can give rise to non-isomorphic $LR$-degree structures.   

Nies \cite{Nie05} gave the following definition in the context of Martin-L\"of randomness with respect to the Lebesgue measure.  We say that
$A$ is \emph{$LR$-reducible} to $B$, denoted $A\LRr B$ if 
\[
\MLR^B\subseteq\MLR^A.
\]
The intuitive idea is that $B$ is more powerful than $A$ as an oracle, as $B$ de-randomizes more sequences than $A$ does. We say that $A$ is \emph{$LR$-equivalent} to $B$, denoted $A\LRe B$ if and only if $A\leq_{LR}B$ and $B\LRr A$.  The $LR$-equivalence classes are called \emph{$LR$-degrees}; the collection of $LR$-degrees is denoted $\LRD$.  Like the Turing degrees $\D_T$, the $LR$-degrees form an uncountable upper semilattice.  However, unlike the structure of the Turing degrees ($\D_T,\leq_T$), the structure of ($\LRD,\LRr$) is not well-understood.

We can extend the definition of $\LRr$ to Martin-L\"of randomness with respect to any $\mu\in\M_c$ as follows.  For $\mu\in\M_c$ and $A,B\in\cs$, we say that $A$ is $LR(\mu)$-reducible to $B$, denoted $A\LRrmu B$ if
\[
\MLR_\mu^B\subseteq\MLR_\mu^A.
\]
Similarly, we can define the $LR(\mu)$-degrees, denoted $\LRDmu$, just as we defined the $LR$-degrees above.  Interestingly, for certain choices of $\mu\in\M_c$, the structure ($\LRD,\LRr$) is very simple.  Let us consider some examples.

\begin{example}
Let $\mu\in\M_c$ be such that $\atom_\mu=\MLR_\mu$ and hence~$\mu(\atom_\mu)=1$.  Then 
$\LRDmu$ consists of a single equivalence class, consisting of all of $\cs$.  The reason is that if $\mu(\{X\})>0$, then $X\in\MLR_\mu^A$ for every $A\in\cs$.
\end{example}

\begin{example}
If $\mu$ is the measure induced by the tally functional $\Phi_A$ for \linebreak $A\in\Delta^0_2\cap\MLR$ (as in the example from Section \ref{sec:tally}), then $\LRDmu$ consists of exactly two elements.  
If we set $A^*:=\Phi_A(A)$, then by Theorem \ref{thm:delta2}($ii$),
\[
\MLR_\mu=\{A^*\}\cup\atom_\mu
\]
where $A^*$ is not computable.  Since $\Phi_A^{-1}(A^*)=\{A\}$, by Theorem \ref{thm:rel-lk},
\[
A\in\MLR^B \Leftrightarrow A^*\in\MLR_\mu^B
\]
 for every $B\in\cs$.  Then there are exactly two $LR(\mu)$-degrees:
\begin{itemize}
\item[] $\mathbf{0}=\{B:A\in\MLR^B\}$,
\item[] $\mathbf{1}=\{B:A\notin\MLR^B\}$.
\end{itemize}
\end{example}

\begin{example}
Let $A\oplus B\in\MLR\cap\Delta^0_2$, and let $\Phi_A$ and $\Phi_B$ be the tally functionals defined in terms of $\Delta^0_2$ approximations of $A$ and $B$, respectively. Next, let $\mu_0$ and $\mu_1$ be the measures induced by $\Phi_A$ and $\Phi_B$, respectively.  By Theorem \ref{thm:delta2}($ii$), 
\[
\MLR_{\mu_0}=\{\Phi_A(A)\}\cup\atom_{\mu_0} \text{ and } \MLR_{\mu_1}=\{\Phi_B(B)\}\cup\atom_{\mu_1}.
\]

If we set
$\nu:=\dfrac{\mu_0+\mu_1}{2}$, by Lemma \ref{lem:sum-meas} we have
\[
\MLR_\nu=\MLR_{\mu_0}\cup\MLR_{\mu_1}=\{\Phi_A(A),\Phi_B(B)\}\cup\atom_{\mu_0}\cup\atom_{\mu_1}.
\]
Clearly, $\nu$ is trivial.  There are exactly four $LR(\nu)$-degrees; namely $\mathbf{0},\mathbf{a},\mathbf{b}$, and $\mathbf{1}$, where (using Theorem \ref{thm:rel-lk} and van Lambalgen's theorem)
\begin{itemize}
\item[] $\mathbf{0}=\{X:A\in\MLR^X\aand B\in\MLR^X \}$,
\item[] $\mathbf{a}=\{X:A\in\MLR^X\aand B\notin\MLR^X \}$,
\item[] $\mathbf{b}=\{X:A\notin\MLR^X\aand B\in\MLR^X \}$,
\item[] $\mathbf{1}=\{X:A\notin\MLR^X\aand B\notin\MLR^X \}$.
\end{itemize} 
In particular, we have $\mathbf{0}<\mathbf{a}<\mathbf{1}$ and
$\mathbf{0}<\mathbf{b}<\mathbf{1}$, but $\mathbf{a}$ and $\mathbf{b}$ are incomparable.  Thus, $\LRDnu$ is isomorphic to the finite Boolean algebra on two atoms, pictured in Figure \ref{fig:diamond}.
\end{example}

\begin{figure}[h!tb]
  \begin{center}
       \centerline{ \includegraphics[scale=.75]{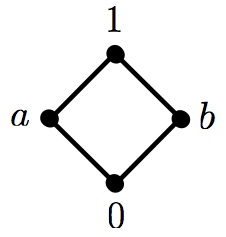}}
    \caption{The finite Boolean algebra on two atoms}
    \label{fig:diamond}
  \end{center}
\end{figure}

In the previous three examples we have defined trivial measures $\mu$ such that the associated $LR(\mu)$-degrees are isomorphic to the finite Boolean algebra of one, two, and four elements, respectively.  Thus, it is natural to consider whether there is such a measure for every finite Boolean algebra.  

\begin{theorem}\label{thm-ba}
For every finite Boolean algebra $\B=(B,\leq)$, there is a trivial measure $\mu\in\M_c$ such that 
\[
(\LRDmu,\LRrmu) \cong (B,\leq).
\]
\end{theorem}

\begin{proof}  First, we fix some notation.  Let $\Deg_{LR(\mu)}(X)=\{Y:X\equiv_{LR(\mu)}Y\}$, and set
\[
\Deg_{LR(\mu)}(X)\leq\Deg_{LR(\mu)}(Y)
\] if and only $X\LRrmu Y$. Now, we proceed in four steps:\\

\noindent
\textbf{Step 1:}  If $n$ is the number of atoms of $\B$, choose $A_1, A_2,\dotsc,A_n\in\MLR\cap\Delta^0_2$ such that for each $J\subseteq\{1,\dotsc, n\}$, if 
\[
X_J=\bigoplus_{j\in J}A_j
\]
then
\[
A_i\in\MLR^{X_J}
\]
for every $i\notin J$. (This can be accomplished using Theorem \ref{thm:bigoplus}.)\\

\medskip
\noindent
\textbf{Step 2:} For each $A_i$, let $\Phi_{A_i}$ be the tally functional defined in terms of a fixed $\Delta^0_2$ approximation of $A_i$, and define $\mu_i$ to be the measure induced by the tally functional $\Phi_{A_i}$.  Let $A_i^*=\Phi_{A_i}(A_i)$.

\medskip
\noindent
\textbf{Step 3:}  Define $\mu:=\frac{1}{n}\sum_{i=1}^n\mu_i$.  It follows from Lemma \ref{lem:sum-meas} that 
\[
\MLR_\mu=\bigcup_{i=1}^n\MLR_{\mu_i}=\{A_1^*,A_2^*,\dotsc,A_n^*\}\cup\bigcup_{i=1}^n\atom_{\mu_i}.
\]

\medskip
\noindent
\textbf{Step 4:}  We verify that for $J,K\subseteq\{1,\dotsc,n\}$, $\Deg_{LR(\mu)}(X_J)\leq \Deg_{LR(\mu)}(X_K)$ if and only if $J\subseteq K$.  First, note that
\[
\MLR_\mu^{X_J}=\{A_i^*:i\not\in J\}\;\text{and}\;\MLR_\mu^{X_K}=\{A_i^*:i\not\in K\}.
\]
Then 
\begin{equation*}
\begin{split}
\Deg_{LR(\mu)}(X_J)\leq \Deg_{LR(\mu)}(X_K)&\Leftrightarrow\MLR_\mu^{X_K}\subseteq\MLR_\mu^{X_J}\\
&\Leftrightarrow\{A_i^*:i\not\in K\}\subseteq\{A_i^*:i\not\in J\}\\
&\Leftrightarrow J\subseteq K.
\end{split}
\end{equation*}
Thus, $\LRDmu=\{\Deg_{LR(\mu)}(X_J):J\subseteq\{1,\dotsc,n\}\}$ is isomorphic to the powerset of $\{1,\dotsc,n\}$, which is isomorphic to $\B$.
\end{proof}

Theorem \ref{thm-ba} can be improved further:

\begin{theorem}\label{thm-fdl}
For every finite distributive lattice $(L,\leq)$, there is a trivial measure $\mu\in\M_c$ such that 
\[
(\mathscr{D}_{LR(\mu)},\leq_{LR(\mu)}) \cong (L,\leq).
\]
\end{theorem}

The following terminology will be useful in the proof of Theorem \ref{thm-fdl}.  Let $L$ be a finite distributive lattice of $n$ elements.  We will consider $L$ in terms of levels, where Level 1 consists of the top element $1_{L}$, Level 2 consists of the immediate predecessors of $1_{L}$, Level 3 consists of the immediate predecessors of elements of Level 2, and so on.  Since $L$ has size $n$, there are only finitely many levels (in fact, at most $n$ levels), and since it is a lattice, the lowest level consists solely of the bottom element $0_{L}$. 

\begin{remark}\label{rmk}
That every element of $L$ lies on a unique level follows from the fact that every finite distributive lattice has a unique rank function that assigns to each $a\in L$ its height in $L$; see, for instance, \cite[pp. 103-104]{Sta12}.
\end{remark}

Recall that for $a,b\in L$, the \emph{meet} of $a$ and $b$, denoted $a\wedge b$, is the greatest element in $L$ such that $a\geq a\wedge b$ and $b\geq a\wedge b$.  The element $c\in L$ is \emph{meet-reducible} if there are $a,b>c$ such that $a\wedge b=c$, and it is \emph{meet-irreducible} if it is not meet-reducible.  

To prove Theorem \ref{thm-fdl}, the idea is (i) construct a lattice of sets isomorphic to $L$, (ii) use these sets to define a collection of tally functionals, and (iii) define a measure in terms of these tally functionals, which will give rise to an $LR$-structure that is isomorphic to $(L,\leq)$.  Let us first consider an example.

Let $(L,\leq)$ be the finite distributive lattice given below in Figure \ref{fig:fdl}.

\begin{figure}[h!tb]
  \begin{center}
\centerline{\includegraphics[scale=.8]{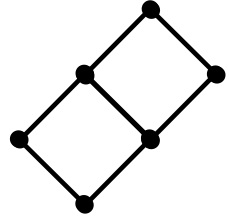}}
    \caption{The finite distributive lattice $(L,\leq)$}
    \label{fig:fdl}
  \end{center}
\end{figure}

\noindent
Now let $A\in\MLR\cap\Delta^0_2$, and let $\{A_i\}_{i\in\omega}$ be such that 
\[
A=\bigoplus_{i\in\omega} A_i,
\]
so that each $A_i\in\MLR\cap\Delta^0_2$.  Hereafter, the sequences $A_0, A_1,\dotsc$ will be referred to as \emph{basic sequences}.  An important feature of these basic sequences is that each $A_i$ is Martin-L\"of random relative to a finite join of any basic sequences that differ from $A_i$ (see Theorem \ref{thm:bigoplus}).

We proceed by associating to each element at each level of $L$ a set consisting of some of the $A_i$'s or joins of the $A_i$'s, yielding a finite distributive lattice of sets that is isomorphic to $L$, as in Figure \ref{fig:fdl2}.

\begin{figure}[h!tb]
  \begin{center}
       \centerline{ \includegraphics[scale=.12]{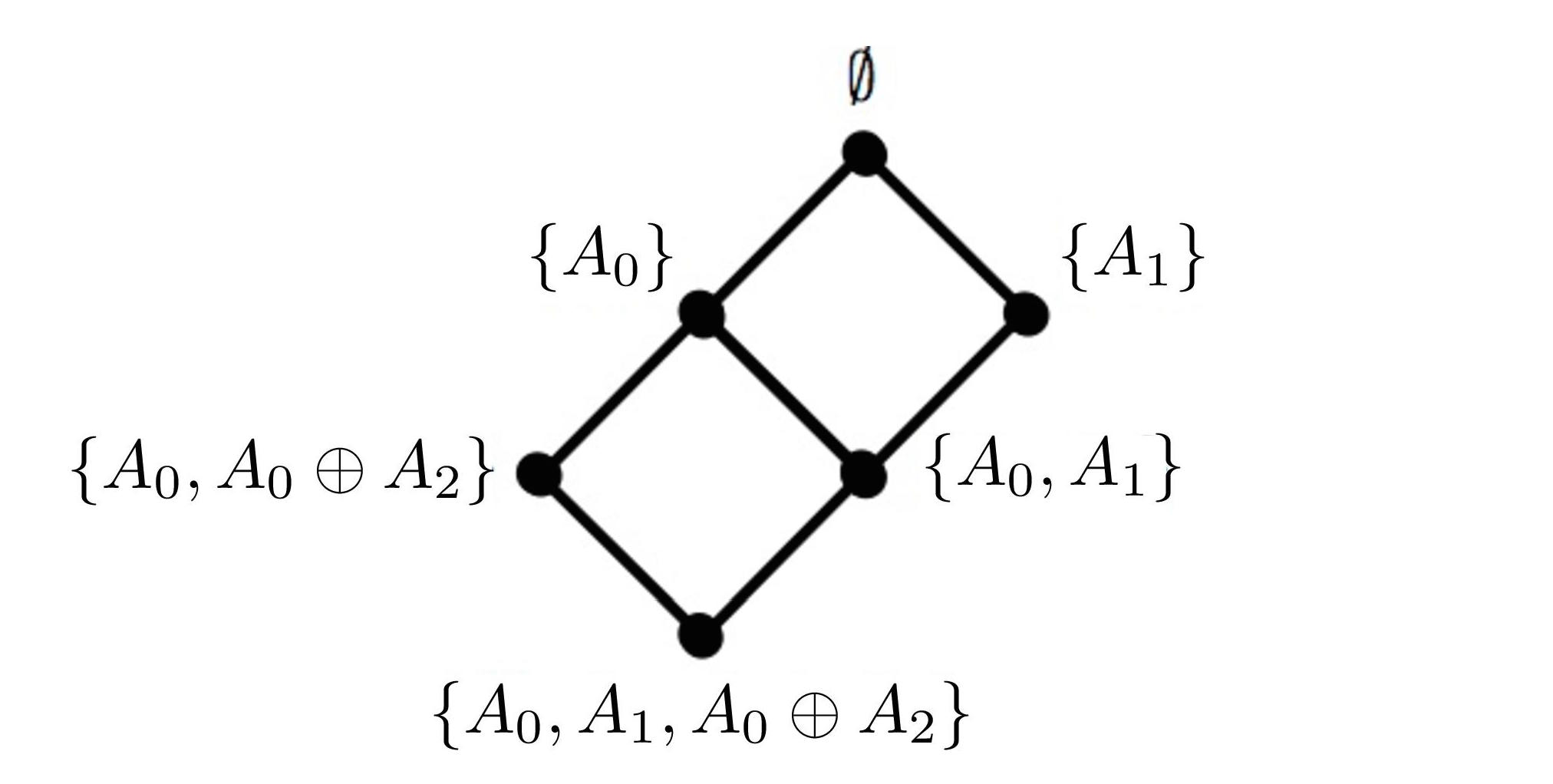}}
    \caption{A finite distributive lattice of sets isomorphic to $L$}
    \label{fig:fdl2}
  \end{center}
\end{figure}

\noindent \textbf{Level 1:}  We associate to the top element $1_{L}$ the empty set.\\

\noindent \textbf{Level 2:}  There are two elements in Level 2, and so we associate to one the set $\{A_0\}$ and to the other $\{A_1\}$.\\

\noindent \textbf{Level 3:}  There are two elements in Level 3, one of which is meet-reducible and the other meet-irreducible.  To the meet-reducible element, we associate the set $\{A_0,A_1\}$, and to the meet-irreducible element (which is below the element associated to the set $\{A_0\}$), we associate the set $\{A_0,A_0\oplus A_2\}$, where $A_2$ is the first basic sequence in $\{A_i\}_{i\in\omega}$ (in the order given by the indices) that has not appeared in the construction thus far.  Note that any sequence that derandomizes $A_0$ also derandomizes $A_0\oplus A_2$, but not every element that derandomizes $A_0\oplus A_2$ also derandomizes $A_0$ (such as $A_2$ itself).\\

\noindent \textbf{Level 4:}  The only element at Level 4 is $0_{L}$, the meet of the two Level 3 elements, and thus we associate to this element the set $\{A_0,A_1,A_0\oplus A_2\}$.\\

Next, for each element in the set associated with $0_{L}$, namely $A_0$, $A_1$, and $A_0\oplus A_2$, let $\mu_{A_0}$, $\mu_{A_1}$, and $\mu_{A_0\oplus A_2}$ be the measures induced by the tally functionals $\Phi_{A_0}$, $\Phi_{A_1}$, and $\Phi_{A_0\oplus A_2}$ defined in terms of fixed $\Delta^0_2$ approximations of $A_0$, $A_1$, and $A_0\oplus A_2$.  We define
\[
\mu:=\frac{1}{3}(\mu_{A_0}+\mu_{A_1}+\mu_{A_0\oplus A_2}),
\]
Let
\[
\Phi_{A_0}(A_0)=A_0^*,\Phi_{A_1}(A_1)=A_1^*,\text{ and }\Phi_{A_0\oplus A_2}(A_0\oplus A_2)=(A_0\oplus A_2)^*.
\]
Then for any $X\in\cs$, 
\[
\MLR_\mu^X=\atom_\mu\cup\S,
\] where $\S$ is equal to one of the following:
\[
\emptyset, \{A_0^*\}, \{A_1^*\}, \{A_0^*,A_1^*\}, \{A_0^*,(A_0\oplus A_2)^*\},\text{ or }\{A_0^*,A_1^*,(A_0\oplus A_2)^*\}.
\]
Thus we have a one-to-one correspondence (that preserves $\subseteq$) between the above sets and those sets associated to the elements of $L$, and thus 
\[
(\LRDmu, \LRrmu)\cong (L,\leq).
\]
Now we proceed in full generality.\\

\begin{proof}[Proof of Theorem \ref{thm-fdl}]
Let $L$ be a finite distributive lattice.  We proceed as in the example.  We first associate basic sequences and joins of basic sequences to elements of the various levels of $L$.\\

\noindent \textbf{Level 1:}  We associate to the top element $1_{L}$ the empty set.\\

\noindent\textbf{Level 2:}  To each of the $j\leq k$ elements in Level 2, we associate a singleton consisting of a basic sequence $A_1,A_2,\dotsc,A_k$.\\

\noindent \textbf{Level $\mathbf{n+1}$:}  The set we associate to a Level $n+1$ element depends on whether it is meet-reducible or meet-irreducible.\\

\begin{itemize}
\item[$\circ$]  The meet-reducible case:  Let $a=b\wedge c$, where $b$ and $c$ are Level $n$ elements. If $\S_b$ is the set of sequences associated to $b$ and $\S_c$ is the set of sequences associated to $c$, then we associate the set $\S_b\cup\S_c$ to $a$.  (Note:  We will have to verify that this is well-defined, for it may be the case that there are Level $n$ elements $b'$ and $c'$ that differ from $b$ and $c$ but also satisfy $b'\wedge c'=a$.)\\

\item[$\circ$]  The meet-irreducible case:  If $a$ is meet-irreducible, then there is only one Level $n$ element such that $a\leq b$.  If $\S_b$ is the set associated to $b$, then we proceed as follows.  First, let $\{A_{i_1},A_{i_2},\dotsc,A_{i_\ell}\}$ be the collection of basic sequences appearing in $S_b$.  That is, these are either elements of $\S_b$ or are contained in joins in $\S_b$ (so that, for instance, the basic sequences appearing in $\{A_0,A_1\oplus A_2\}$ are $A_0,A_1,$ and $A_2$).  Let $N\in\omega$ be the least such that the basic sequence $A_N$ has not appeared in any set associated to an element of $L$.  Then to $a$ we associate the set
\[
\biggl\{\bigoplus_{j=1}^\ell A_{i_j}\oplus A_N\biggr\}\cup \S_b.
\]
\end{itemize}

To verify that $\S_{L}=(\{\S_a:a\in L\},\leq)$ is a finite distributive lattice isomorphic to $(L,\leq)$ (where $\S_a\leq \S_b$ if and only if $\S_a\supseteq\S_b$), we first show that meets in $\S_{L}$ are well-defined.  First, suppose that $a,b,c,d\in L$ are distinct elements such $a=b\wedge c=c\wedge d$, as in Figure \ref{fig:m3-1}.

\begin{figure}[h!tb]
  \begin{center}
       \centerline{ \includegraphics[scale=.8]{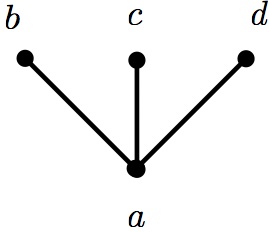}}
    \caption{The case in which $a=b\wedge c=c\wedge d$}
    \label{fig:m3-1}
  \end{center}
\end{figure}

\noindent
We claim that $b\vee c\neq c\vee d$.  For otherwise, the lattice $M_3$ (pictured in Figure \ref{fig:m3-2} below) would be embeddable into $L$, contradicting the fact that $L$ is distributive.

\begin{figure}[h!tb]
  \begin{center}
       \centerline{ \includegraphics[scale=.8]{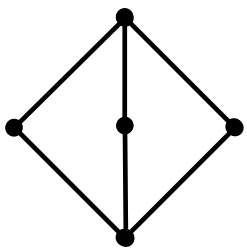}}
    \caption{The lattice $M_3$}
    \label{fig:m3-2}
  \end{center}
\end{figure}

Thus, we are in the situation as depicted by Figure \ref{fig:m3-3}.\\

\begin{figure}[h!tb]
  \begin{center}
       \centerline{ \includegraphics[scale=.8]{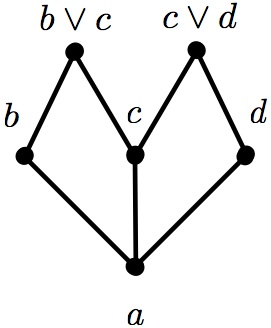}}
    \caption{$b\vee c\neq c\vee d$}
    \label{fig:m3-3}
  \end{center}
\end{figure}

\noindent
We must also have $b\vee d\neq b\vee c$ and $b\vee d\neq c\vee d$, for otherwise $M_3$ would be embeddable into $L$ (for instance, Figure \ref{fig:m3-4} shows the case that $b\vee d=c\vee d$).

\begin{figure}[h!tb]
  \begin{center}
       \centerline{ \includegraphics[scale=.8]{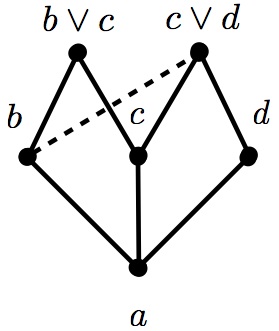}}
    \caption{The case that $b\vee d=c\vee d$}
    \label{fig:m3-4}
  \end{center}
\end{figure}

If $e=b\vee c, f=c\vee d$, and $g=b\vee d$, then $b=e\wedge g, c=e\wedge f$, and $d=f\wedge g$, and thus none of $b,c,$ or $d$ is meet-irreducible.

Now, suppose that 
\begin{equation}\label{eq:bcd}
\S_b\cup \S_c\neq\S_c\cup \S_d
\end{equation}
Since $b=e\wedge g, c=e\wedge f$, and $d=f\wedge g$, the following hold:
\begin{center}
$\S_b=\S_e\cup\S_g$\\
$\S_c=\S_e\cup\S_f$\\
$\S_d=\S_f\cup\S_g.$\\
\end{center}
By (\ref{eq:bcd}), we have
\[
\S_e\cup\S_f\cup\S_g=\S_b\cup \S_c\neq\S_c\cup \S_d=\S_e\cup\S_f\cup\S_g,
\]
which is impossible.  

In the case in which $a$ is a Level $n+1$ element (so that there is maximal chain $a<b_1<\dotsc<b_n=1_L$ of exactly $n$ elements above $a$) and there are distinct Level $n$ elements $b,c,b',c'$ strictly above $a$ such that $a=b\wedge c=b'\wedge c'$, we claim that $a=b\wedge c'$.  Suppose not.  Then $a<(b\wedge c')<b$ so $a$ also has a maximal chain $a<c_1<\dotsc<c_j=1_L$ of $j\geq n+1$ elements above it (since $b$ is a Level $n$ element).  Hence $a$ is also a Level $j+1$ element.  Since $j\neq n$, this contradicts the fact stated in Remark \ref{rmk} that every member of $L$ lies in a unique level.  Thus we have $a=b\wedge c'$.  

Now, since $a=b\wedge c=b\wedge c'=b'\wedge c'$, we are in the previous case.  Applying the argument given above to $b,c,c'$ and then to $b,b',c'$, we conclude that 
\[
\S_b\cup \S_c=\S_{b}\cup \S_{c'}=\S_{b'}\cup\S_{c'}.
\]
Thus, meets are well-defined. 

Next we show that the isomorphism between $\S_{L}=(\{\S_a:a\in L\},\leq)$ and $(L,\leq)$ holds level by level.  In particular, we show that meets and joins are preserved level by level.  First, it is clear that the top two levels of $\S_{L}$ and $L$ are isomorphic.  Now suppose that $\S_{L}$ and $L$ are isomorphic from Level 1 to Level $n$.  Having associated to Level $n$ elements $a$ and $b$ the sets $\S_a$ and $\S_b$, we associate the set $\S_a\cup\S_b$ to $a\wedge b$.

Suppose Level $n+1$ elements $a$ and $b$ are associated with $\S_a$ and $\S_b$.  To show that $\S_{a\vee b}$, the set associated to $a\vee b$, is $\S_a\cap\S_b$, we consider three cases.

\medskip
\noindent
{\bf Case 1:} First, if both $a$ and $b$ are meet-irreducible, then either (i) there is some Level $n$ element $c$ such that $a,b\leq c$, or (ii) there are distinct Level $n$ elements $c$ and $d$ such that $a\leq c$ and $b\leq d$.  

\medskip
\noindent
Subcase 1(i): By the procedure given above, 
\[
\S_a=\S_c\cup\{B\oplus A_\ell\}
\]
and
\[
\S_b=\S_c\cup\{B\oplus A_{\ell'}\},
\]
where $B$ is the join of the basic sequences appearing in $\S_c$ and $A_\ell,A_{\ell'}$ are distinct basic sequences not contained in any set associated to elements of Levels $k\leq n$.  Thus $\S_{a\vee b}=\S_c=\S_a\cap\S_b$.

\medskip
\noindent
Subcase 1(ii).
In this subcase,
\[
\S_a=\S_c\cup\{B_0\oplus A_\ell\},
\]
and
\[
\S_b=\S_d\cup\{B_1\oplus A_{\ell'}\},
\]
where $B_0$ and $B_1$ are the joins of the basic sequences appearing in $\S_c$ and $\S_d$, respectively, and $A_\ell,A_{\ell'}$ are distinct basic sequences not contained in any set associated to elements of Level $k$ for any $k\leq n$.  By induction, there is some $e\in L$ such that $e=c\vee d$ and $\S_e=\S_c\cap\S_d$.  Then we have $e=a\vee b$ and
\[
\S_a\cap\S_b=(\S_c\cup\{B_0\oplus A_\ell\})\cap(\S_d\cup\{B_1\oplus A_{\ell'}\})=\S_c\cap\S_d=\S_e=S_{a\vee b}.
\]

\medskip
\noindent
{\bf Case 2:} If $a$ is meet-irreducible but $b$ is meet-reducible, then again there are two subcases to consider:  Either (i) there is some Level $n$ element $c$ such that $a,b\leq c$, or (ii) there are distinct Level $n$ elements $c,d,e$ such that $a\leq c$ and $b=d\wedge e$.  

\medskip
\noindent
Subcase 2(i):  We have
\[
\S_a=\S_c\cup\{B\oplus A_\ell\},
\]
where $B$ is the join of the basic sequences appearing in $\S_c$ and $A_\ell$ is a basic sequence not contained in any set associated to any element of Level $k$ for any $k\leq n$, and
\[
\S_b=\S_c\cup\S_d
\]
for some Level $n$ element $d\neq c$.  Again it follows that
\[
\S_a\cap\S_b=(\S_c\cup\{B\oplus A_\ell\})\cap(\S_c\cup\S_d)=\S_c=\S_{a\vee b}.
\] 

\medskip
\noindent
Subcase 2(ii):  In this subcase, $c\vee(d\wedge e)=a\vee b$.  As above, 
\[
\S_a=\S_c\cup\{B\oplus A_\ell\}
\]
and
\[
\S_b=\S_d\cup\S_e.
\]
By the inductive hypothesis, we have $\S_c\cap(\S_d\cup\S_e)=\S_{c\vee(d\wedge e)}$, and thus
\[
\S_a\cap\S_b=(\S_c\cup\{B\oplus A_\ell\})\cap(\S_d\cup\S_e)=\S_c\cap(\S_d\cup\S_e)=\S_{c\vee(d\wedge e)}=\S_{a\vee b}.
\]

\medskip
\noindent
{\bf Case 3:}  Lastly, in the case that $a$ and $b$ are both meet-reducible, either (i) there are distinct Level $n$ elements $c,d,e$ such $a=c\wedge d$, and $b=d\wedge e$ or (ii) there are distinct Level $n$ elements $c,d,e,f$ such that $a=c\wedge d$ and $b=e\wedge f$.  

\medskip
\noindent
Subcase 3(i):  Since $a,b\leq d$, it follows that $\S_a=\S_c\cup\S_d$, $\S_b=\S_d\cup\S_e$, $\S_c\cap\S_e=\emptyset$, and thus
\[
\S_a\cap\S_b=(\S_c\cup\S_d)\cap(\S_d\cup\S_e)=\S_d=\S_{a\vee b}.
\]

\medskip
\noindent
Subcase 3(ii):  Note that $a\vee b=(c\wedge d)\vee(e\wedge f)$.  Since  $\S_a=\S_c\cup\S_d$ and $\S_b=\S_e\cup\S_f$, by the inductive hypothesis, it follows that
\[
\S_a\cap\S_b=(\S_c\cup\S_d)\cap(\S_e\cup\S_f)=\S_{(c\vee d)\wedge(e\vee f)}=\S_{a\vee b}.
\]

Having verified that $\S_{L}$ is a finite distributive lattice, we now turn to defining the trivial measure $\mu$.  Let
\[
\{B_1,\dotsc, B_k\}
\]
be the set in $\S_{L}$ associated to $0_{L}$.  By our construction, each $B_i$ is either a basic sequence or the join of some basic sequences.  Furthermore, since the basic sequences are all in $\MLR\cap\Delta^0_2$ and each is Martin-L\"of random relative to the finite join of any number of basic sequences that differ from it, it follows from van Lambalgen's Theorem  that each $B_i\in\MLR\cap\Delta^0_2$.

Let $\Phi_i$ be the tally functional defined in terms of the $\Delta^0_2$ approximation of $B_i$, and let $B^*_i=\Phi_i(B_i)$, so that $\MLR_{\mu_i}=\{B^*_i\}\cup\atom_{\mu_i}$.  Given $\S$, one of the sets of sequences that is associated to some element of $L$, we define $\S^*$ such that
\[
B\in\S\;\text{if and only if}\;B^*\in\S^*.
\]
Setting 
\[
\mu:=\frac{1}{k}\sum_{i=1}^k\mu_i,
\]
we claim that $(\LRDmu,\LRrmu)\cong(\S_{L},\leq)$.  First we show that for each $\S^*_a$, there is some $X\in\cs$ such that
\[
\S^*_a\cup\atom_\mu=\MLR^X_\mu.
\]
If we let $\RS(B)=\{X\in\cs:B\in\MLR^X\}$ be the \emph{randomness scope} of $B$, note that by Theorem \ref{thm:rel-lk},
\[
B_i\in\MLR^X\Leftrightarrow B^*_i\in\MLR_\mu^X,
\]
and hence $X\in\RS(B_i)$ if and only if $B_i^*\in\MLR^X_\mu$.
Observe that $\lambda(\RS(B))=1$ for every $B\in\MLR$, since by van Lambalgen's Theorem, $\MLR^B\subseteq\RS(B)$ and $\lambda(\MLR^B)=1$.

If
\[
\S^*_a=\{B^*_{i_1},\dotsc, B^*_{i_j}\},
\]
then
\[
\X=\bigcap_{j=1}^k\RS(B_{i_j})\neq\emptyset,
\]
since the finite intersection of measure one sets has measure one.  Thus for any $X\in\X$, we have $\MLR^X_\mu=\S^*_a\cup\atom_\mu$.

We claim that for each $X\in\cs$, there is some $a\in L$ such that $\MLR^X_\mu=\S^*_a\cup\atom_\mu$.  Let $\{A_1,\dotsc, A_k\}$ be the collection of basic sequences appearing in the elements of $\S_{L}$.  For each $i\leq k$, let $a_i\in L$ be the element such that the basic sequence $A_i$ first appears in $\S_{a_i}$.  Furthermore, for $j\leq k$, let $\{A_1,\dotsc,A_j\}$ be the basic sequences that make up the singletons assigned to Level 2 elements of $L$ (which we'll call the \emph{Level Two basic sequences}), and let $\{A_{j+1},\dotsc,A_k\}$ be the basic sequences that added when we assign sets to meet-irreducible elements of $L$ (which we'll call the \emph{meet-irreducible basic sequences}).  For each $X\in\cs$, there is some $J\subseteq\{1,\dotsc,j\}$ such that
\[
X\in\bigcap_{i\in J}\RS(A_i)\aand X\notin\bigcap_{i\in\{1,\dotsc,j\}\setminus J}\RS(A_i).
\]
That is, $J$ picks out the indices of the Level Two basic sequences that $X$ fails to derandomize.  Now if $J=\emptyset$, then as every $B_i\in\S_L$ is either a Level Two basic sequence or is the join of a Level Two basic sequence with some other sequence, it follows that $\MLR^X_\mu=\atom_\mu$.  If $J\neq\emptyset$, then it follows from the construction that
\[
\{A^*_i:i\in J\}\cup\atom_\mu=(\S_{\bigwedge_{i\in J}a_i})^*\cup\atom_\mu\subseteq\MLR_\mu^X.
\]
Turning to the meet-irreducible basic sequences, there is some $K\subseteq\{j+1,\dotsc,k\}$ such that
\[
X\in\bigcap_{i\in K}\RS(A_i)\aand X\notin\bigcap_{i\in\{j+1,\dotsc,k\}\setminus K}\RS(A_i).
\]
Now it may be that in the course of the construction, some $A_i$ with $i\in K$ is joined to some $A_\ell$ with $\ell\notin J$ (or joined to some sequence with $A_\ell$ as a subsequence).  This occurs when we associate a collection of sequences to a meet-irreducible element of $L$ that is below the element of $L$ to which we associated the Level Two basic sequence $A_\ell$. Let
\[
\widehat K=K\setminus\{i\in K: \S_{a_i}\supseteq \S_{a_\ell}\;\text{for some}\;\ell\notin J \}.
\]
Then if $\widehat{K}=\emptyset$, then $\MLR^X_\mu=(\S_{\bigwedge_{i\in J}a_i})^*\cup\atom_\mu$.  Otherwise, setting
\[
a=\bigwedge_{i\in J}a_i\wedge\bigwedge_{i\in\widehat{K}}a_i,
\]
it follows from the construction that $\MLR^X_\mu=\S^*_a\cup\atom_\mu$.

Thus, every $\S^*_a$ is the collection of non-atoms in $\MLR^X_\mu$ for some $X\in\cs$, and for every $X\in\cs$, there is some $\S^*_a$ such that $\MLR^X_\mu=\S^*_a\cup\atom_\mu$.  Since $X\LRrmu Y$ if and only if $\MLR_\mu^Y\subseteq\MLR_\mu^X$, it follows that 
\[
(\LRDmu,\LRrmu)\cong(\S_{L},\leq).
\]
\end{proof}

We conclude with two questions.

\begin{question}
If $(L,\leq)$ is an infinite, computable, distributive lattice, is there a trivial measure $\mu\in\M_c$ such that  
\[
(\mathscr{D}_{LR(\mu)},\leq_{LR(\mu)}) \cong (L,\leq)?
\]
\end{question}

\begin{question}
Is there an example of a finite non-distributive lattice $(L,\leq)$ and a trivial measure $\mu\in\M_c$ such that  
\[
(\mathscr{D}_{LR(\mu)},\leq_{LR(\mu)}) \cong (L,\leq)?
\]
\end{question}


\section*{Acknowledgements}
The author would like to thank Laurent Bienvenu, Peter Cholak, and Damir Dzhafarov for helpful conversations on the material in this article (which appeared
in the author's dissertation).  The author would also like to thank Paul Shafer for useful feedback on a preliminary draft of this article. 

\bibliographystyle{alpha}
\bibliography{trivialmeasures.bib}

\begin{thebibliography}{DNWY06}

\bibitem[BP12]{BiePor12}
Laurent Bienvenu and Christopher Porter.
\newblock Strong reductions in effective randomness.
\newblock {\em Theoret. Comput. Sci.}, 459:55--68, 2012.

\bibitem[DH10]{DowHir10}
Rodney~G. Downey and Denis~R. Hirschfeldt.
\newblock {\em Algorithmic randomness and complexity}.
\newblock Theory and Applications of Computability. Springer, New York, 2010.

\bibitem[DNWY06]{DowNieWeb06}
Rod Downey, Andr\'e Nies, Rebecca Weber, and Liang Yu.
\newblock Lowness and {$\Pi^0_2$} nullsets.
\newblock {\em J. Symbolic Logic}, 71(3):1044--1052, 2006.

\bibitem[Kau91]{Kau91}
Steven Kautz.
\newblock {\em Degrees of Random Sets}.
\newblock PhD thesis, Cornell University, 1991.

\bibitem[Kur81]{Kur81}
Stuart Kurtz.
\newblock Randomness and genericity in the degrees of unsolvability.
\newblock Ph.D. Thesis, University of Illinois at Urbana, 1981.

\bibitem[ML66]{Mar66}
Per Martin-L{\"o}f.
\newblock The definition of random sequences.
\newblock {\em Information and Control}, 9:602--619, 1966.

\bibitem[Nie05]{Nie05}
Andr{\'e} Nies.
\newblock Lowness properties and randomness.
\newblock {\em Adv. Math.}, 197(1):274--305, 2005.

\bibitem[Nie09]{Nie09}
Andr{\'e} Nies.
\newblock {\em Computability and randomness}.
\newblock Oxford Logic Guides. Oxford University Press, 2009.

\bibitem[NST05]{NieSteTer05}
Andr{\'e} Nies, Frank Stephan, and Sebastiaan~A. Terwijn.
\newblock Randomness, relativization and {T}uring degrees.
\newblock {\em J. Symbolic Logic}, 70(2):515--535, 2005.

\bibitem[Sch71]{Sch71}
Claus-Peter Schnorr.
\newblock {\em Zuf\"alligkeit und {W}ahrscheinlichkeit. {E}ine algorithmische
  {B}egr\"undung der {W}ahrscheinlichkeitstheorie}.
\newblock Lecture Notes in Mathematics, Vol. 218. Springer-Verlag, Berlin,
  1971.

\bibitem[Sch77]{Sch77}
Claus-Peter Schnorr.
\newblock A survey of the theory of random sequences.
\newblock In {\em Basic problems in methodology and linguistics ({P}roc.
  {F}ifth {I}nternat. {C}ongr. {L}ogic, {M}ethodology and {P}hilos. of {S}ci.,
  {P}art {III}, {U}niv. {W}estern {O}ntario, {L}ondon, {O}nt., 1975)}, pages
  193--211. Univ. Western Ontario Ser. Philos. Sci., Vol. 11. Reidel,
  Dordrecht, 1977.

\bibitem[Soa87]{Soa87}
Robert~I. Soare.
\newblock {\em Recursively enumerable sets and degrees}.
\newblock Perspectives in Mathematical Logic. Springer-Verlag, Berlin, 1987.
\newblock A study of computable functions and computably generated sets.

\bibitem[Sta12]{Sta12}
Richard~P. Stanley.
\newblock {\em Enumerative combinatorics. {V}olume 1}, volume~49 of {\em
  Cambridge Studies in Advanced Mathematics}.
\newblock Cambridge University Press, Cambridge, second edition, 2012.

\bibitem[vL90]{Lam90}
Michiel van Lambalgen.
\newblock The axiomatization of randomness.
\newblock {\em J. Symbolic Logic}, 55(3):1143--1167, 1990.

\bibitem[ZL70]{ZvoLev70}
A.~K. Zvonkin and L.~A. Levin.
\newblock The complexity of finite objects and the basing of the concepts of
  information and randomness on the theory of algorithms.
\newblock {\em Uspehi Mat. Nauk}, 25(6(156)):85--127, 1970.

\end{thebibliography}

\end{document}